\documentclass[a4paper,12pt]{article}

\usepackage{fullpage}
\usepackage{mathtools}
\usepackage[utf8]{inputenc}
\usepackage{amsthm,amsmath,amstext, amssymb, xcolor, scrextend, tikz, multirow, enumerate, wasysym, makecell, tabularx, hyperref,  pbox, lipsum,csquotes,  tikz-cd, multicol, calc, accents, stackengine, kpfonts}
\usepackage[all]{xy}
\usepackage{xcolor}
\usepackage{palatino}
\usepackage{cancel}
\usepackage{authblk}

\usepackage{mathtools}

\DeclareSymbolFont{AMSb}{U}{msb}{m}{n}
\DeclareMathSymbol{\N}{\mathbin}{AMSb}{"4E}
\DeclareMathSymbol{\Z}{\mathbin}{AMSb}{"5A}
\DeclareMathSymbol{\R}{\mathbin}{AMSb}{"52}
\DeclareMathSymbol{\Q}{\mathbin}{AMSb}{"51}
\DeclareMathSymbol{\I}{\mathbin}{AMSb}{"49}
\DeclareMathSymbol{\C}{\mathbin}{AMSb}{"43}

\DeclareFontFamily{U}{mathx}{\hyphenchar\font45}
\DeclareFontShape{U}{mathx}{m}{n}{<-> mathx10}{}
\DeclareSymbolFont{mathx}{U}{mathx}{m}{n}
\DeclareMathAccent{\widebar}{0}{mathx}{"73}

\def\lim{\mathop{\rm lim}\nolimits}

\def\AA{\mathbb A}

\newcommand{\res}[1]{\hspace{-0.6mm}\downarrow_{\hspace{-0.25mm}{#1}}}

\newcommand{\tn}{\textnormal}
\newcommand{\F}{\mathbb{F}}

\newcommand{\iso}{\cong}

\numberwithin{equation}{section}
\newtheorem{theorem}{Theorem}[section]

\newtheorem{lemma}[theorem]{Lemma}

\theoremstyle{remark}
\newtheorem{remark}[theorem]{Remark}
\theoremstyle{definition}
\newtheorem{defn}[theorem]{Definition}
\newtheorem{example}[theorem]{Example}

\usetikzlibrary{decorations.shapes}
\usetikzlibrary{decorations.text}
\usetikzlibrary{calc}

\hypersetup{colorlinks=true,linkcolor=[rgb]{0.5, 0.0, 0},citecolor=[rgb]{0.5, 0.0, 0.5}, filecolor=magenta,urlcolor=blue}

\title{Butler's method applied to $\Z_p[C_p\times C_p]$-permutation modules}

\author[a]{John W.\ MacQuarrie}
\author[a]{Marlon Stefano}
\affil[a]{Universidade Federal de Minas Gerais, Belo Horizonte, MG, Brazil}

\begin{document}

\footnotetext{\textit{Email addresses:} john@mat.ufmg.br (John MacQuarrie), mestanislau@ufmg.br (Marlon Stefano)}

\maketitle

\begin{abstract}
Let $G$ be a finite $p$-group with normal subgroup $N$ of order $p$.  The first author and Zalesskii have previously given a characterization of permutation modules for $\Z_pG$ in terms of modules for $G/N$, but the necessity of their conditions was not known.  We apply a correspondence due to Butler to demonstrate the necessity of the conditions, by exhibiting highly non-trivial counterexamples to the claim that if both the $N$-invariants and the $N$-coinvariants of a given lattice $U$ are permutation modules, then so is $U$.  
\end{abstract}

\section{Introduction}

Let $R$ be a complete discrete valuation ring whose residue field has characteristic $p$ ($R$ will almost always be $\Z_p$ or $\F_p$) and let $G$ be a finite $p$-group.  An $RG$-module is a \emph{lattice} if it is free of finite $R$-rank.  An $RG$-lattice is a \emph{permutation module} if it possesses an $R$-basis that is set-wise preserved by multiplication from $G$.  Given an $RG$-module $U$ and a normal subgroup $N$ of $G$, the \emph{$N$-invariants} of $U$, $U^N$, is the largest submodule of $U$ on which $N$ acts trivially, and the \emph{$N$-coinvariants} of $U$, $U_N$, is the largest quotient module of $U$ on which $N$ acts trivially.

\medskip

In \cite{MacquarrieZalesskiiDM}, extending the famous sufficient condition of Weiss \cite[Theorem 2]{WeissAnnals} to a characterization, the first author and Pavel Zalesskii proved the following:

\begin{theorem}\label{DM theorem}
Let $G$ be a finite $p$-group, $U$ a $\Z_pG$-lattice and $N$ a normal subgroup of $G$ of order $p$.  Then $U$ is a permutation module if, and only if
\begin{enumerate}
    \item\label{condition inv and coinv} $U^N$ and $U_N$ are permutation $\Z_p[G/N]$-modules and 
    \item $(U/U^N)_N$ is a permutation $\F_p[G/N]$-module.
\end{enumerate}
\end{theorem}

Simple examples are given in \cite{MacquarrieZalesskiiDM} showing that we cannot remove the conditions of Part 1.  It was however left open whether Condition 2 is redundant.  That is, whether $U$ is a permutation module if, and only if, both $U^N$ and $U_N$ are permutation modules.  We will show here that Condition 2 cannot be removed, by exhibiting examples of lattices $U$ for $\Z_p[C_p\times C_p]$ ($p$ odd) and for $\Z_2[C_2\times C_2\times C_2]$ for which both $U^N$ and $U_N$ are permutation modules, but $U$ is not.  The examples seem to be highly non-trivial.

We believe that the methods used to construct these examples are interesting in their own right, utilizing in a new way a wonderful correspondence due to Butler \cite{ButlerKlein4, ButlerAbelian}, wherein he associates to a lattice what he calls a ``diagram''.  For Butler, 
the idea was to analyze the \emph{category} of lattices for a given abelian $p$-group, using the correspondence to compare it with a category that is better understood.  The novelty of our approach is to use Butler's correspondence to analyze the structure of a \emph{given} lattice, by analyzing how prescribed properties of the lattice 
affect the corresponding diagram.  In particular, we will characterize the diagrams of those $\Z_p[C_p\times C_p]$-lattices $U$ for which $U$, $U^N$ and $U_N$ are permutation modules.  Using these characterizations, we find a diagram whose corresponding lattice $U$ has the properties that both $U^N$ and $U_N$ are permutation modules, but $U$ is not.  In the future, we believe that analyses of this type will yield further insights in the study of permutation modules.

\section{Preliminaries}

Notations introduced here will be used throughout our discussion.  In what follows, $G$ is always a finite abelian $p$-group and $N$ 

is a subgroup of $G$ of order $p$.

Given a $\Z_pG$-lattice $U$, the $\Z_p[G/N]$-modules of $N$-invariants $U^N$ and $N$-coinvariants $U_N$ are, respectively, the largest submodule and the largest quotient module of $U$ on which $N$ acts trivially.  Explicitly
\begin{align*}
    U^N & = \{u\in U\,|\, nu = u\,\,\forall n\in N\} \\
    U_N & = U/I_NU 
\end{align*}
where $I_N$, the augmentation ideal of $\Z_pN$, is the ideal generated by the kernel of the map $\Z_pN\to \Z_p$ sending $\sum_{n\in N} \lambda_nn$ to $\sum \lambda_n$.

By \cite[Theorem 2.6]{HellerReinerAnnals} there are three isomorphism classes of indecomposable $\Z_pN$-lattice, being the trivial module $\Z_p$, the free module $\Z_pN$ and a non-permutation module $S$ of $\Z_p$-rank $p-1$, which can be described in either of the following equivalent ways:
$$S = I_N = \Z_pN/(\Z_pN)^N.$$

\begin{lemma}[{\cite[Lemma 8]{MacquarrieZalesskiiDM}}]\label{lemma UresN perm iff N coinv lattice}
Let $H$ be a cyclic group of order $p$ and $U$ a $\Z_pH$-lattice.  Then $U_H$ is a lattice if, and only if, $U$ is a permutation module.
\end{lemma}

If ever $H$ is a subgroup of $G$, we use the symbol $\widehat{H}$ to denote the element $\sum_{h\in H}h$ in $\Z_pG$.

\begin{lemma}\label{lemma multiply sums formulas}
Let $G = C_p\times C_p$ and $H,K$ subgroups of $G$.  In $\Z_pG$ we have $\widehat{H}\cdot \widehat{K} = |H\cap K|\cdot\widehat{HK}$.
\end{lemma}

\begin{proof}
This follows by direct calculations.

\end{proof}

\section{Butler's Correspondence for $\Z_p[C_p\times C_p]$}

We describe the parts we require from \cite{ButlerAbelian} for the special case of $\Z_p[C_p\times C_p]$, which is what we require to construct examples for $p$ odd.  The case of $p=2$ must be treated separately, but the methods are similar and we only present the example, in Section \ref{section p equals 2 case}.

In what follows, $G = C_p\times C_p = N\times C$ 
.  We denote as usual by $\Q_p, \F_p$ the field of fractions and the residue field of the $p$-adic integers $\Z_p$, respectively.  The group algebra $\Q_pG = \Q_p\otimes_{\Z_p}\Z_pG$ has a decomposition 
$$\Q_pG = e_0\Q_pG \times \prod_{\substack{H\leqslant G \\ |H|=p}} {e_{H}\Q_pG},$$
where the $e_i$ are the $p+2$ primitive orthogonal idempotents of $\Q_pG$, given as
$$e_0 = \frac{1}{p^2}\widehat{G}\,,\quad e_H = \frac{1}{p^2}(p\widehat{H} - \widehat{G}).$$
Denote by $I$ the indexing set $\{0, H_1, \hdots, H_{p+1}\}$ of the primitive idempotents, where the $H_i$ are the subgroups of $G$ of order $p$.   
For each $i\in I$ we denote by $\Lambda_{(i)}$ both the $\Z_pG$-algebra $e_i\Z_pG$ and the free $e_i\Z_pG$-module of rank 1\footnote{we write $\Lambda_{(i)}$ rather than $\Lambda_i$  to avoid confusion with coinvariants.}.  Recall that $\Z_pG$ is a local algebra, with maximal ideal $I_G + p{\Z_pG}$.

\begin{lemma}
\label{lemma LambdaH facts}
Let $H, K$ be distinct subgroups of order $p$.
\begin{enumerate}
    \item 
    $\widehat{H}\cdot e_K = 0$,
    \item $\Lambda_{(K)}$ has no non-zero $H$-fixed points,
    \item $J(\Lambda_{(K)}) = I_H\Lambda_{(K)}$.
\end{enumerate}
\end{lemma}

\begin{proof}
\begin{enumerate}
    \item By Lemma \ref{lemma multiply sums formulas},
    $$\widehat{H}\cdot e_K = \frac{1}{p^2}(p\widehat{H}\widehat{K} - \widehat{H}\widehat{G}) = \frac{1}{p^2}(p\widehat{G} - p\widehat{G}) = 0.$$
    \item $\widehat{H}$ acts as multiplication by $p$ on a fixed point of $\Lambda_{(K)}$, and so ${\Lambda_{(K)}}^H=0$ by Part 1.
    
    \item $p\Lambda_{(K)}\subseteq I_H\Lambda_{(K)}$ because by Part 1
    $$pe_K = pe_K - \widehat{H}e_K = \sum_{h\in H}(1-h)e_K\in I_H\Lambda_{(K)}.$$
\end{enumerate}
\end{proof}

\begin{defn}
The $\Z_pG$-lattice $U$ is \emph{reduced} if it has neither free direct summands nor summands isomorphic to any $\Lambda_i$.  
\end{defn}

\begin{defn}[{\cite[\S 2.2]{ButlerAbelian}}]
A \emph{diagram} for $C_p\times C_p$ is a $(p+3)$-tuple $V_* = (V; V_{(i)}\,,\, i\in I)$, where 
\begin{itemize}
    \item $V$ is a finitely generated $\F_p[C_p\times C_p]$-module,
    \item the $V_{(i)}$ are $\F_p[C_p\times C_p]$-submodules of $V$, \item for each $j\in I$ we have $\sum_{i\neq j}V_{(i)} = V$,
    \item for each $i\in I$, the induced action of $\Z_p[C_p\times C_p]$ on $V_{(i)}$ factors through the canonical surjection $\Z_p[C_p\times C_p]\to \Lambda_{(i)}$ sending $1$ to $e_i$. 
\end{itemize}
A \emph{map of diagrams} $V_*\to V'_*$ is an $\F_p[C_p\times C_p]$-module homomorphism $\rho:V\to V'$ such that $\rho(V_{(i)})\subseteq V'_{(i)}$ for each $i\in I$.
\end{defn}

We make explicit how one obtains a diagram from a reduced lattice, and vice-versa, fixing as we do some notation that will remain in force throughout the rest of the article.  As we follow Butler's arguments in \cite{ButlerAbelian}, simply specialising to the case of $\Z_p[C_p\times C_p]$ (which simplifies things slightly), we don't give any justifications.

\begin{itemize}
    \item \textbf{Getting a diagram from a reduced lattice:} Let $U$ be a reduced lattice.  We regard $U$ as the sublattice $1\otimes U$ of the $\Q_pG$-module $\Q_pU = \Q_p\otimes_{\Z_p}U$.  Defining for each $i\in I$ the submodule $U_{(i)} := e_iU$ of $\Q_pU$, the sum 
$$U_* := \sum_{i\in I} U_{(i)}$$
is direct, and since $U$ is reduced we have
$$pU_*\subsetneq U\subsetneq U_*.$$
Defining
$$V = U_*/U\,,\quad V_{(i)} = (U_{(i)} + U)/U,$$
the tuple $V_* = (V ; V_{(i)}\,,\, i\in I)$ is the diagram corresponding to $U$.

\item \textbf{Getting a reduced lattice from a diagram:} Let $V_* = (V ; V_{(i)}\,,\, i\in I)$ be a diagram.  For each $i\in I$, let $F_{(i)}$ be the direct sum of $\textnormal{dim}_{\F_p}(V_{(i)}/\tn{Rad}(V_{(i)}))$ copies of $\Lambda_{(i)}$ and denote by $F$ the direct sum $\bigoplus_{i\in I}F_{(i)}$.  For each $i$ let $f_{i} : F_{(i)}\to V_{(i)}$ be a surjective homomorphism of $\Lambda_{(i)}$-modules inducing an isomorphism $F_{(i)}/\textnormal{Rad}(F_{(i)})\to V_{(i)}/\textnormal{Rad}(V_{(i)})$.  We thus obtain for each $i$ a map $F_{(i)}\xrightarrow{f_{i}}V_{(i)}\hookrightarrow V$ and hence by summing these maps, a map $f : F\to V$.  The reduced lattice corresponding to $V_*$ is $\textnormal{Ker}(f)$.
\end{itemize}

These processes are mutually inverse in the sense that if $U'$ is the reduced lattice obtained from the diagram of $U$, then $U'\iso U$ as $\Z_pG$-lattices, and if $V'_*$ is the diagram coming from the reduced lattice associated to $V_*$, then $V'_*\iso V_*$.

\medskip

It follows that given a reduced lattice $U$, we may associate a unique (up to isomorphism) diagram $V_*$ and from this, the modules $F_{(i)}, F$, with $U$ being considered as a submodule of $F$.  We will use these notations freely in what follows.

\section{Characterizations using diagrams}

Throughout this section, $U$ is a reduced $\Z_pG$-lattice with diagram $V_* = (V; V_{(i)}\,,\,i\in I)$ and $F_{(i)}, F$ are the corresponding modules defined in the previous section.  Given an element $x\in F$, we denote by $\overline{x}$ its image in $V$.  Throughout, the letters $H,K$ will always refer to subgroups of order $p$.  If an element of $\bigoplus F_{(i)}$ is written like $x = \sum_{H\neq N}x_H$, this should be taken to mean that $x\in \bigoplus_{\substack{H\leqslant G \\ |H|=p\\H\neq N}} F_{(H)}$ with each $x_H\in F_{(H)}$, and similarly in $V$.  

\begin{theorem}\label{prop U perm char}
The reduced lattice $U$ is a permutation module if, and only if, 
$G$ acts trivially on each $V_{(i)}$ and $V = V_{(0)} = \bigoplus_{\substack{H\leqslant G \\ |H|=p}}V_{(H)}$.
\end{theorem}

\begin{proof}
A direct calculation shows that the reduced permutation lattice $\Z_p[G/H] = \langle u\rangle$ has diagram
$$(\F_p; \F_p, 0, \hdots, 0, \F_p, 0, \hdots, 0)$$
where the last non-zero entry occurs in the position indexed by $H$: taking $K$ to be any subgroup of order $p$ different from $H$ we have by Lemma \ref{lemma multiply sums formulas}:
\begin{align*}
    e_0\cdot u & = \frac{1}{p^2}\widehat{G}u = \frac{1}{p^2}\widehat{K}\widehat{H}u = \frac{1}{p}\widehat{K}u \\
    e_H\cdot u & = \frac{1}{p^2}(p\widehat{H} - \widehat{K}\widehat{H})u = \frac{1}{p}(p\widehat{1} - \widehat{K})u \\
    e_K\cdot u & = \frac{1}{p^2}(p\widehat{K} - \widehat{H}\widehat{K})u = \frac{1}{p^2}(p\widehat{1} - \widehat{H})\widehat{K}u = \frac{1}{p^2}(p\widehat{1} - p\widehat{1})\widehat{K}u = 0.
\end{align*}
Since $e_0\cdot u + e_H\cdot u = u\in U$, $U_*/U$ has dimension 1 and the claim follows.  A reduced permutation module is a direct sum of modules $\Z_p[G/H]$ for different $H$ of order $p$, and the diagram of a direct sum is the direct sum of the diagrams.  The result follows.
\end{proof}

\begin{lemma}\label{lemma properties of invariants}
\begin{enumerate}
    \item $U^N = U\cap (F_{(0)} + F_{(N)})$
    \item For any subgroup $H\neq N$ of order $p$ we have $I_HU^N\subseteq U\cap F_{(N)}$, with equality if, and only if, $U^N$ is a permutation module.
\end{enumerate}
\end{lemma}

\begin{proof}
\begin{enumerate}
    \item $N$ acts trivially on $F_{(0)}, F_{(N)}$, while $F_{(H)}$ has no $N$-fixed points for $H\neq N$ by Lemma \ref{lemma LambdaH facts}.  Part 1 now follows because $F$ is the direct sum of the $F_{(i)}$.
    
    \item The inclusion follows from Part 1 because $I_HF_{(0)} = 0$.  For any $u\in U\cap F_{(N)}$ we have $\widehat{H}u = \widehat{H}e_Nu = 0$ by Lemma \ref{lemma LambdaH facts} and so
    $$pu = pu - \widehat{H}u = \sum_{h\in H}(1-h)u\in I_HU^N.$$
    If $U^N$ is a permutation module, then $U^N/I_HU^N$ is a lattice by Lemma \ref{lemma UresN perm iff N coinv lattice} and so $u\in I_HU^N$.  On the other hand, if $I_HU^N = U\cap F_{(N)}$ then 
    $$U^N/I_HU^N = U^N/(U^N\cap F_{(N)}),$$
    which is a lattice since $F_{(N)}$ is a summand of $F$, and so $U^N$ is a permutation module by Lemma \ref{lemma UresN perm iff N coinv lattice} again.
\end{enumerate}
\end{proof}

\begin{theorem}\label{prop invariants perm char}
The lattice $U^N$ is a permutation module if, and only if, ${V_{(N)}}^G\subseteq V_{(0)}$.
\end{theorem}

\begin{proof}
Suppose first that $U^N$ is a permutation module and fix $\overline{x}\in {V_{(N)}}^G$, with lift $x\in F_N$.  Since $(1-c)\overline{x} = 0$, we have $(1-c)x\in U\cap F_{(N)}$.  But by Lemma \ref{lemma properties of invariants}, $U\cap F_{(N)} = I_CU^N$ so by Part 1 of Lemma \ref{lemma properties of invariants} there are $x_0\in F_{(0)}, x_N\in F_{(N)}$ with $x_0+x_N\in U$ and $(1-c)x = (1-c)(x_0 + x_N) = (1-c)x_N$.  As $(1-c)(x-x_N)=0$ and $F_{(N)}$ has no non-zero $C$-fixed points, $x=x_N$, so that $x_0 + x \in U$.  Hence  $\overline{x} = - \overline{x_0}\in V_{(0)}$, as required.

Suppose now that ${V_{(N)}}^G\subseteq V_{(0)}$.  To conclude that $U^N$ is a permutation module, it is enough to check that $U\cap F_{(N)}\subseteq I_CU^N$, by Lemma \ref{lemma properties of invariants}.  So fix $y\in U\cap F_{(N)}$.  Since $U$ is reduced, $y\in \tn{Rad}(F_{(N)}) = I_CF_{(N)}$, by Lemma \ref{lemma LambdaH facts}, so we may write $y = (1-c)x$ with $x\in F_{(N)}$.  Since $(1-c)\overline{x} = 0$, $\overline{x}\in {V_{(N)}}^G$, so in $V_{(0)}$ by hypothesis.  Thus there is $x_0\in F_{(0)}$ with $\overline{x_0} = \overline{x}$, hence $x - x_0\in U^N$ and so $y = (1-c)(x-x_0)\in I_CU^N$ as required.
\end{proof}

\begin{lemma}\label{lemma UN lattice upstairs char}
We have
$$I_NU\subseteq U \cap \Big(\bigoplus_{H\neq N}F_{(H)}\Big)$$
with equality if, and only if, $U_N$ is a lattice.
\end{lemma}

\begin{proof}
The inclusion is clear because $(1-n)(F_{(0)} + F_{(N)}) = 0$.  Suppose that $U_N$ is a lattice and take $x$ in the right hand set.  Since $\widehat{N}x = 0$ by Lemma \ref{lemma LambdaH facts}, $px = \sum_{n\in N} (1-n)x\in I_NU$ and hence $x\in I_NU$ since $U_N$ is a lattice.  On the other hand, the right hand set is a pure submodule of $U$ because $\bigoplus_{H\neq N}F_{(H)}$ is a summand of $F$, and hence if this set is equal to $I_NU$ then $U_N$ is a lattice.
\end{proof}

We need a technical lemma:

\begin{lemma}\label{lemma technical exists map}
Fix $a\in I$ and $B\subseteq I\backslash\{a\}$.  The homomorphism 
\begin{align*}
    \varphi : U\cap \bigoplus_{i\in B\cup\{a\}}F_{(i)} & \to V_{(a)}\cap \sum_{i\in B}V_{(i)} \\
    \sum x_i \quad& \mapsto \quad \overline{x_a}
\end{align*}
is well-defined and surjective.
\end{lemma}

\begin{proof}
The element $\overline{x_a}$ is in $V_{(a)}\cap \sum_{i\in B}V_{(i)}$ because $\sum x_i\in U$, and hence $\overline{x_a} = -\sum_{i\in B}\overline{x_i}$.  Given $v\in V_{(a)}\cap \sum_{i\in B}V_{(i)}$, write $v = -\sum_{i\in B}v_i$ and lift $v$ to $x_a$ in $F_{(a)}$ and each $v_i$ to $x_i$ in $F_{(i)}$.  The element $\sum_{i\in B\cup\{a\}}x_i\in U\cap \bigoplus_{i\in B\cup\{a\}}F_{(i)}$ and maps to $v$ by $\varphi$.
\end{proof}

Recall that by Lemma \ref{lemma UresN perm iff N coinv lattice}, $U_N$ is a lattice if, and only if, the restriction $U\res{N}$ of $U$ to $N$ is a permutation module.  Thus the following theorem is also a characterization of those reduced $U$ for which $U\res{N}$ is a permutation module.

\begin{theorem}\label{prop coinvariants lattice char}
The module $U_N$ is a lattice if, and only if
\begin{enumerate}
    \item For each $H\neq N$ of order $p$, $V_{(H)}\cap (\sum_{K\neq N,H}V_{(K)})\subseteq I_NV_{(H)}$, and
    \item $V^N\subseteq V_{(0)} + V_{(N)}$.
\end{enumerate}
\end{theorem}

\begin{proof}
Suppose first that $U_N$ is a lattice.  We check the first condition, so fix an element $v\in V_{(H)}\cap (\sum_{K\neq N,H}V_{(K)})$.  By Lemma \ref{lemma technical exists map} there is $u = \sum_{K\neq N}u_K \in U\cap \bigoplus_{K\neq N}F_{(K)}$ such that $\overline{u_H} = v$.  By Lemma \ref{lemma UN lattice upstairs char}, there is  $u' = \sum_{i\in I}u'_i\in U$ such that $u=(1-n)u' = \sum (1-n)u'_i$.  As the sum of the $F_{(i)}$ is direct, we have $u_H = (1-n)u_H'$ and hence
$$v = \overline{u_H} = (1-n)\overline{u_H'}\in I_NV_{(H)}.$$
To show the second condition, it is enough to confirm that $(\sum_{H\neq N}V_{(H)})^N\subseteq V_{(0)} + V_{(N)}$ so fix $v = \sum_{H\neq N}\overline{x_H}\in V^N$ and a lift $\sum_{H\neq N}x_H$ of $v$.  Then $(1-n)\sum_{H\neq N}x_H\in U\cap \bigoplus_{H\neq N}F_{(H)}$, which is $I_NU$ by Lemma \ref{lemma UN lattice upstairs char} and so
$$(1-n)\sum_{H\neq N}x_H = (1-n)\sum_{i\in I}y_i$$
with $\sum y_i\in U$.  As $F_{(H)}$ has no non-zero $N$-fixed points, $x_H = y_H$ for each $H\neq N$.  Hence
$$ V_{(0)} + V_{(N)}\ni -(\overline{y_0} + \overline{y_N}) = \sum_{H\neq N}\overline{y_H} = \sum_{H\neq N}\overline{x_H} = v.$$

Suppose conversely that Conditions $1$ and $2$ hold for $U$.  We will check the conditions of Lemma \ref{lemma UN lattice upstairs char}.  Given $u = \sum x_H\in U\cap \bigoplus_{H\neq N}F_{(H)}$, Lemma \ref{lemma technical exists map} and Condition 1 imply that $\overline{x_H}\in I_NV_{(H)} = \tn{Rad}(V_{(H)})$ for each $H\neq N$ and hence, since $F_{(H)}/\tn{Rad}(F_{(H)})\to V_{(H)}/\tn{Rad}(V_{(H)})$ is an isomorphism, $x_H\in \tn{Rad}(F_{(H)}) = I_NF_{(H)}$.  For each $H\neq N$, write $x_H = (1-n)y_H$ for some $y_H\in F_{(H)}$. Then $(1-n)\sum_{H\neq N}\overline{y_H} = 0$ so that $\sum_{H\neq N}\overline{y_H}\in V^N$ and hence  
$$\sum_{H\neq N}\overline{y_H} = -\overline{y_0} - \overline{y_N}\in V_{(0)}+V_{(N)}$$ 
for some $y_0\in F_{(0)}$ and $y_N\in F_{(N)}$, by Condition 2.  Thus $y = \sum_{i\in I}y_i\in U$ and $$I_NU \ni (1-n)y = \sum_{H\neq N}x_H = u.$$
\end{proof}

\begin{lemma}\label{lemma UN perm upstairs char}
Let $U$ be a reduced lattice such that $U_N$ is a lattice.  Then 
$$I_GU\subseteq U \cap \Big(\bigoplus_{H}F_{(H)}\Big)$$
with equality if, and only if, $U_N$ is a permutation module.
\end{lemma}

\begin{proof}
The inclusion is clear because $(1-g)F_{(0)}=0$ for every $g\in G$.  Since $U_N$ is a lattice, Lemma \ref{lemma UresN perm iff N coinv lattice} says that $(U_N)_C = U/I_GU$ is a lattice if, and only if, $U_N$ is a permutation module.  Suppose that $U_N$ is a permutation module and fix $u = \sum_{H}x_H\in U \cap \bigoplus_{H}F_{(H)}$.  By Lemma \ref{lemma multiply sums formulas}, $(p\cdot 1 - \widehat{K})u = \sum_{H\neq K}px_H\in I_GU$, and hence, summing over the subgroups of order $p$ we get that $\sum_H p^px_H = p^pu\in I_GU$, and hence $u\in I_GU$ since $U/I_GU$ is a lattice.  On the other hand, the right hand submodule is pure in $U$ because $\bigoplus_H F_{(H)}$ is a summand of $F$, so if the sets are equal then $I_GU$ is pure in $U$ and $U_N$ is a permutation module.
\end{proof}

\begin{remark}
Lemma \ref{lemma UN perm upstairs char} is false without the condition that $U_N$ be a lattice:  the given equality applies for the $\Z_2[C_2\times C_2]$-lattice $U$ whose diagram has $V = \F_2\oplus \F_2, V_{(0)} = \langle (1,1)\rangle, V_{(N)} = 0, V_{(C)} = \langle(1,0)\rangle$ and $V_{(\langle nc\rangle)} = \langle(0,1)\rangle$.  Here $U_N$ is not a lattice.

\end{remark}

\begin{theorem}\label{prop coinvariants perm char}
Let $U$ be a reduced lattice such that $U_N$ is a lattice.  Then $U_N$ is a permutation module if, and only if
$V_{(H)} \cap \sum_{K\neq H}V_{(K)}\subseteq I_GV_{(H)}$ for every subgroup $H$ of order $p$.
 
\end{theorem}

\begin{proof}
Suppose that $U_N$ is a permutation module.  Given an element $\overline{x_H} = \sum_{K\neq H}\overline{x_K}$ in $V_{(H)} \cap \sum_{K\neq H}V_{(K)}$,
then $x_H - \sum_{K\neq H}x_K \in U \cap \Big(\bigoplus_{H}F_{(H)}\Big)$, hence in $I_GU\subseteq I_GF$ by Lemma \ref{lemma UN perm upstairs char}.  Thus $x_H\in I_GF_{(H)}$ and hence $\overline{x_H}\in I_GV_{(H)}$.

\medskip

Suppose now that $V_{(H)} \cap \sum_{K\neq H}V_{(K)}\subseteq I_GV_{(H)}$ for every $H$.  By Lemma \ref{lemma UN perm upstairs char}, to prove that $U_N$ is a permutation module, it is enough to consider an element $u\in U\cap \bigoplus_K F_{(K)}$ and check that it is in $I_GU$.  The given inclusion implies that $U\cap (\bigoplus_{K}F_{(K)})\subseteq I_G(\bigoplus_{K}F_{(K)})$, so we can write
$$u = (1-c)y_N - (1-n)\sum_{K\neq N}y_K,
$$
so that $(1-c)\overline{y_N} = (1-n)\sum_{K\neq N}\overline{y_K}$.  But by the definition of a diagram, we have $V = V_{(0)} + \sum_{K\neq N}V_{(K)}$ and so we can write $\overline{y_N} = \overline{x_0} + \overline{x_C} + \sum_{K\neq N,C}\overline{x_K}$, and hence
$$(1-c)\overline{y_N} = (1-c)\sum_{K\neq N,C}\overline{x_K}.$$
Thus the element $w := (1-c)y_N - (1-c)\sum_{K\neq N,C}x_K$ is in $U$, and indeed in $I_CU$ because $y_N - \sum_{K\neq N,C}x_K - x_0 - x_C\in U$ and 
$$w = (1-c)(y_N - x_0 - x_C -  \sum_{K\neq N,C}x_K).$$
On the other hand, since
$$(1-n)\sum_{K\neq N}\overline{y_K} = (1-c)\overline{y_N} = (1-c)\sum_{K\neq N,C}\overline{x_K},$$
the element $w' := (1-n)\sum_{K\neq N}y_K - (1-c)\sum_{K\neq N,C}x_K \in U\cap \bigoplus_{K\neq N}F_{(K)}$, so $w'\in I_NU$ by Lemma \ref{lemma UN lattice upstairs char} since $U_N$ is a lattice.  It follows that $u = w-w'\in I_CU + I_NU = I_GU$, as required.
\end{proof}

\section{Examples for $C_p\times C_p$}

If $p=2$, it follows from Theorem \ref{DM theorem} that a $\Z_2[C_2\times C_2]$-lattice $U$ is a permutation module if, and only if, both $U^N$ and $U_N$ are permutation modules, because every $\F_2[G/N]$-module is a permutation module so that Condition 2 is automatically satisfied.  But this is not the case for odd primes.  Here we present a counterexample for each odd prime $p$, using the results from the previous section.

\begin{example}
Fix an odd prime $p$ and let $G = C_p\times C_p = \langle n\rangle\times \langle c\rangle$.  We construct a diagram for $\Z_pG$.  Let $V$ be the $\F_p[C_p\times C_p]$-module with $\F_p$-basis $v_1, \hdots, v_{p+2}$ and action from $G$ given as follows:
$$n\cdot v_1 = v_1, \quad n\cdot v_2 = v_2, \quad n\cdot v_j = v_1 + v_j\quad\forall j\geqslant 3,$$
$$c\cdot v_1 = v_1, \quad c\cdot v_2 = v_1 + v_2, \quad c\cdot v_3 = v_3,\quad c\cdot v_j = -k_jv_1 + v_j\quad\forall j\geqslant 4,$$
where $k_j\in \{1, \hdots, p-1\}$ is such that $(j-3)k_j = 1$ modulo $p$.  The submodules $V_{(i)}$ are defined as follows:
\begin{align*}
    V_{(0)} & = \langle v_1,\, v_2 - v_3 + v_4,\, (k_{j+1}-k_j)v_2 - v_j + v_{j+1}\hbox{ for } 4\leqslant j\leqslant p+1\rangle,\\
    V_{(N)} & = \langle v_1, v_2\rangle,\\
    V_{(C)} & = \langle v_1, v_3\rangle,\\
    V_{(\langle nc^a\rangle)} & = \langle v_1, v_{a+3}\rangle, \qquad a\in \{1,\hdots,p-1\}
\end{align*}
The tuple $(V ; V_{(i)}\,,\, i\in I)$ is a diagram, 
with corresponding reduced lattice $U$.  Then
\begin{itemize}
    \item $U^N$ is a permutation module by Theorem \ref{prop invariants perm char}, because ${V_{(N)}}^G = \langle v_1\rangle \subseteq V_{(0)}$.
    
    \item $U_N$ is a lattice by Theorem \ref{prop coinvariants lattice char}: the first condition is satisfied because $V_{(H)}\cap \sum_{K\neq N,H}V_{(K)} = \langle v_1\rangle$ and $v_1 = (n-1)v_j$ for any $j\geqslant 3$.  The second condition is satisfied because the element $\sum_{j=1}^{p+2}\lambda_jv_j\in V^N$ if, and only if, $p\mid \sum_{j=3}^{p+2}\lambda_j$, and hence $V^N$ is generated by $v_1, v_2$ and the elements $v_j - v_{j+1}$ for $j\in \{3, \hdots, p+1\}$, which are in $V_{(0)} + V_{(N)}$. 
    
    \item $U_N$ is a permutation module by Theorem \ref{prop coinvariants perm char}: $V_{(H)}\cap \sum_{K\neq H}V_{(K)} = \langle v_1\rangle$ which is in $I_NV_{(H)}$ if $H\neq N$ and, since $v_1 = (c-1)v_2$, is in $I_CV_{(H)}$ if $H = N$.
    
    \item But $U$ is not a permutation module by Theorem \ref{prop U perm char}, since $G$ does not act trivially on $V$.
\end{itemize}
The lattice $U$ has $\Z_p$-rank $p^2 + p - 1$.  In the smallest case, when $p=3$, it has $\Z_3$-rank $11$ and with respect to some $\Z_3$-basis of $U$, the matrices of the actions of $n$ and $c$ respectively, are
$$\begin{psmallmatrix}
1 & 0 & -1 & -1 & -1 & -1 & 0 & 0 & 0 & 0 & 0 \\
0 & 1 & 0 & 0 & 0 & 0 & 0 & 0 & 0 & 0 & 0 \\
0 & 0 & -2 & 0 & 0 & -1 & 1 & 0 & 1 & 0 & 0 \\
0 & 0 & 0 & -2 & 0 & -1 & -1 & 0 & 0 & 1 & 0 \\
0 & 0 & 0 & 0 & -2 & -1 & 0 & 0 & 0 & 0 & 1 \\
0 & 0 & 0 & 0 & 0 & 1 & 0 & 0 & 0 & 0 & 0 \\
0 & 0 & 0 & 0 & 0 & 0 & 1 & 0 & 0 & 0 & 0 \\
0 & 0 & 0 & 0 & 0 & 0 & 0 & 1 & 0 & 0 & 0 \\
0 & 0 & -3 & 0 & 0 & -1 & 1 & 0 & 1 & 0 & 0 \\
0 & 0 & 0 & -3 & 0 & -1 & -1 & 0 & 0 & 1 & 0 \\
0 & 0 & 0 & 0 & -3 & -1 & 0 & 0 & 0 & 0 & 1
\end{psmallmatrix}\hbox{ and }\begin{psmallmatrix}
1 & -1 & 0 & 1 & -1 & 0 & 0 & 0 & 0 & -1 & 0 \\
0 & -2 & 0 & 0 & 0 & 0 & -1 & 1 & 0 & 0 & 0 \\
0 & 0 & 1 & 0 & 0 & 0 & 0 & 0 & 0 & 0 & 0 \\
0 & 0 & 0 & 1 & 0 & 0 & 0 & 0 & 0 & -1 & 0 \\
0 & 0 & 0 & 0 & -2 & -1 & 0 & 0 & 0 & 0 & 1 \\
0 & 0 & 0 & 0 & 0 & 1 & 0 & 0 & 0 & 0 & 0 \\
0 & 0 & 0 & 0 & 0 & 0 & 1 & 0 & 0 & 0 & 0 \\
0 & -3 & 0 & 0 & 0 & 0 & -1 & 1 & 0 & 0 & 0 \\
0 & 0 & 0 & 0 & 0 & 0 & 0 & 0 & 1 & 0 & 0 \\
0 & 0 & 0 & 3 & 0 & 1 & 1 & 0 & 0 & -2 & 0 \\
0 & 0 & 0 & 0 & -3 & -1 & 0 & 0 & 0 & 0 & 1
\end{psmallmatrix}.
$$
\end{example}

\section{Example for $C_2\times C_2\times C_2$}\label{section p equals 2 case}

A similar application of Butler's Correspondence to the group $G = C_2\times C_2\times C_2 = \langle n\rangle \times \langle b\rangle\times \langle c\rangle$ yields the following lattice $U$ for which $U^N$ and $U_N$ are permutation modules, but $U$ is not.  The matrices of $n,b,c$ respectively, are:
$$\begin{psmallmatrix}
1 & 0 & 0 & 0 & 0 & 0 & 0 & 0 & 0 & 0 & 0 \\
0 & 1 & 0 & 0 & 0 & 0 & 0 & 0 & 0 & 0 & 0 \\
0 & 0 & 1 & 0 & 0 & 0 & 0 & 0 & 0 & 0 & 0 \\
0 & 0 & 0 & 1 & 0 & 0 & 0 & 0 & 0 & 0 & 0 \\
0 & 1 & 0 & 0 & 1 & 0 & 0 & 1 & 1 & 1 & 1 \\
0 & 1 & 0 & 0 & 0 & 1 & 0 & 1 & 1 & 1 & 1 \\
0 & -1 & 0 & 0 & 0 & 0 & 1 & -1 & -1 & -1 & -1 \\
0 & -1 & 0 & 0 & 0 & 0 & 0 & -1 & 0 & 0 & 0 \\
0 & 0 & -1 & 1 & 0 & 0 & 0 & 0 & -1 & 0 & 0 \\
0 & -1 & 1 & 0 & 0 & 0 & 0 & 0 & 0 & -1 & 0 \\
0 & 0 & 0 & -1 & 0 & 0 & 0 & 0 & 0 & 0 & -1
\end{psmallmatrix}\,,\,\begin{psmallmatrix}
1 & 0 & 0 & 0 & 0 & 0 & 0 & 0 & 0 & 0 & 0 \\
0 & 1 & 0 & 0 & 0 & 0 & 0 & 0 & 0 & 0 & 0 \\
0 & 0 & 1 & 0 & 0 & 0 & 0 & 0 & 0 & 0 & 0 \\
0 & 0 & 0 & 1 & 0 & 0 & 0 & 0 & 0 & 0 & 0 \\
0 & 0 & 0 & 1 & 0 & 1 & 0 & 0 & 0 & 1 & 1 \\
0 & 1 & -1 & 0 & 1 & 0 & 0 & 0 & 0 & 1 & 1 \\
1 & 0 & 1 & -1 & -1 & -1 & -1 & 0 & 0 & -1 & -1 \\
0 & 0 & 0 & 0 & 0 & 0 & 0 & 1 & 0 & 0 & 0 \\
0 & 0 & 0 & 0 & 0 & 0 & 0 & 0 & 1 & 0 & 0 \\
0 & -1 & 1 & 0 & 0 & 0 & 0 & 0 & 0 & -1 & 0 \\
0 & 0 & 0 & -1 & 0 & 0 & 0 & 0 & 0 & 0 & -1
\end{psmallmatrix}\,,\,\begin{psmallmatrix}
1 & 0 & 0 & 0 & 0 & 0 & 0 & 0 & 0 & 0 & 0 \\
0 & 1 & 0 & 0 & 0 & 0 & 0 & 0 & 0 & 0 & 0 \\
0 & 0 & 1 & 0 & 0 & 0 & 0 & 0 & 0 & 0 & 0 \\
0 & 0 & 0 & 1 & 0 & 0 & 0 & 0 & 0 & 0 & 0 \\
0 & 0 & 0 & 1 & 0 & 0 & 1 & -1 & 0 & -1 & 0 \\
1 & 1 & 0 & 0 & -1 & -1 & -1 & -1 & 0 & -1 & 0 \\
0 & 0 & 0 & -1 & 1 & 0 & 0 & 1 & 0 & 1 & 0 \\
0 & 0 & 0 & 0 & 0 & 0 & 0 & 1 & 0 & 0 & 0 \\
0 & 0 & -1 & 1 & 0 & 0 & 0 & 0 & -1 & 0 & 0 \\
0 & 0 & 0 & 0 & 0 & 0 & 0 & 0 & 0 & 1 & 0 \\
0 & 0 & 0 & -1 & 0 & 0 & 0 & 0 & 0 & 0 & -1
\end{psmallmatrix}.$$

\bibliographystyle{alpha}
\bibliography{references}

\end{document}